\documentclass{amsart}
\usepackage{latexsym,bbm,amsxtra,amscd,ifthen}
\usepackage{amsfonts}
\usepackage{verbatim}
\usepackage{amsmath}
\usepackage{amsthm}
\usepackage{amssymb}
\usepackage{url}
\usepackage[arrow,matrix]{xy}
\usepackage[colorlinks=true,citecolor=blue,linkcolor=blue]{hyperref}

\theoremstyle{plain}

\newtheorem{theorem}{Theorem}
\newtheorem{lemma}[theorem]{Lemma}
\newtheorem{proposition}[theorem]{Proposition}
\newtheorem{corollary}[theorem]{Corollary}

\numberwithin{theorem}{section}
\numberwithin{equation}{theorem}

\theoremstyle{definition}
\newtheorem{definition}[theorem]{Definition}

\newtheorem{example}[theorem]{Example}
\newtheorem{remark}[theorem]{Remark}

\newtheorem*{question*}{Question}

\DeclareMathOperator{\Ext}{Ext}

\DeclareMathOperator{\Hom}{Hom}

\DeclareMathOperator{\id}{injdim}

\DeclareMathOperator{\gldim}{gldim}

\DeclareMathOperator{\depth}{depth}

\DeclareMathOperator{\gr}{gr}

\DeclareMathOperator{\mcm}{mcm}

\DeclareMathOperator{\uExt}{{\underline{Ext}}}
\DeclareMathOperator{\uHom}{{\underline{Hom}}}

\def\kk{\mathbbm{k}}
\def\mmod{\operatorname{mod}}

\def\gr{\operatorname{gr}}
\def\tor{\operatorname{tor}}
\def\Gr{\operatorname{Gr}}

\def\qgr{\operatorname{qgr}}

\begin{document}

\title[Clifford deformations]
{Clifford deformations of Koszul Frobenius algebras and noncommutative quadrics}

\author{Ji-Wei He and Yu Ye}

\address{He: Department of Mathematics,
Hangzhou Normal University,
Hangzhou Zhejiang 311121, China}
\email{jwhe@hznu.edu.cn}
\address{Ye: School of Mathematical Sciences, University of Science and Technology of China, Hefei Anhui 230026, China}
\email{yeyu@ustc.edu.cn}

\begin{abstract} Let $E$ be a Koszul Frobenius algebra. A Clifford deformation of $E$ is a finite dimensional $\mathbb Z_2$-graded algebra $E(\theta)$, which corresponds to a noncommutative quadric hypersurface $E^!/(z)$, for some central regular element $z\in E^!_2$. It turns out that the bounded derived category $D^b(\gr_{\mathbb Z_2}E(\theta))$ is equivalent to the stable category of the maximal Cohen-Macaulay modules over $E^!/(z)$ provided that $E^!$ is noetherian. As a consequence, $E^!/(z)$ is a noncommutative isolated singularity if and only if the corresponding Clifford deformation $E(\theta)$ is a semisimple $\mathbb Z_2$-graded algebra. The preceding equivalence of triangulated categories also indicates that Clifford deformations of trivial extensions of a Koszul Frobenius algebra are related to the Kn\"{o}rrer Periodicity Theorem for quadric hypersurfaces. As an application, we recover Kn\"{o}rrer Periodicity Theorem without using matrix factorizations.
\end{abstract}

\subjclass[2010]{16S37, 16E65, 16G50}


\keywords{Koszul Frobenius algebra, Clifford deformation, noncommutative quadric hypersurface, maximal Cohen-Macaulay module}


\maketitle


\setcounter{section}{-1}
\section{Introduction}

Let $S$ be a noetherian Koszul Artin-Schelter regular algebra, and let $z\in S_2$ be a central regular element of $S$. The quotient algebra $A=S/(z)$ is a Koszul Artin-Schelter Gorenstein algebra, which is usually called a quadric hypersurface algebra. Smith and Van den Bergh introduced in \cite{SvdB} a finite dimensional algebra $C(A)$ associated to the quadric hypersurface algebra $A$, which determines the representations of the singularities of $A$. In particular, the simplicity of $C(A)$ implies the smoothness of $\text{Proj}A$ (cf. \cite[Proposition 5.2]{SvdB}), where $\text{Proj} A$ is the noncommutative projective scheme (cf. \cite{AZ}).

Let $E=S^!$ be the quadratic dual algebra of $S$. Then $E$ is a Koszul Frobenius algebra. The dimension of $C(A)$ is equal to the one of the even degree part of $E$ (cf. \cite[Lemma 5.1]{SvdB}). As the key observation of this paper, we notice that $C(A)$ may be obtained from some Poincar\'{e}-Birkhoff-Witt (PBW) deformation of $E$.

Write the above Koszul Frobenius algebra $E$ as $E=T(V)/(R)$ for some finite dimensional vector space $V$ over a field $\kk$ and $R\subseteq V\otimes V$. Let $\theta\colon R\to \kk$ be a linear map. If $(\theta\otimes 1-1\otimes \theta)(V\otimes R\cap R\otimes V)=0$, then $\theta$ defines a PBW deformation $E(\theta)$ of $E$ (cf. \cite[Proposition 1.1, Chapter 5]{PP}); that is, $E(\theta)$ is a filtered algebra whose associated graded algebra is isomorphic to $E$. We call $\theta$ a {\it Clifford map} and $E(\theta)$ a {\it Clifford deformation} of $E$ (more precisely, see Definition \ref{def1}). First examples are classical Clifford algebras, which motivate the name of Clifford deformation. In fact, if we take $E$ to be the exterior algebra generated by $V$, then any Clifford map corresponds to a symmetric bilinear form on $V$, and the associated Clifford deformation can realize the classical Clifford algebras (cf. Example \ref{ex1}).

For every Clifford map $\theta$, the algebra $E(\theta)$ has a natural $\mathbb Z_2$-graded structure. Note that every Clifford map $\theta$ corresponds to a central element $z\in S_2$ (cf. Remark \ref{rem2}).
It turns out that the degree 0 part $E(\theta)_0$ coincides with the finite dimensional algebra $C(A)$ if $z$ is regular (cf. Proposition \ref{prop2}). Hence the structure of $E(\theta)$ will determine the representations of singularities of $A$.

Let $\gr_{\mathbb Z_2}E(\theta)$ be the category of finite dimensional right $\mathbb Z_2$-graded $E(\theta)$-modules. Let $\underline{\mcm}A$ be the stable category of graded maximal Cohen-Macaulay modules over $A$. The main observation of the paper is the following result (cf. Theorem \ref{thm1}(iii)).

\begin{theorem}\label{thm0.1} Let $S$ be a noetherian Koszul Artin-Schelter regular algebra and let $z\in S_2$ be a central regular element. Set $A=S/(z)$ and let $E=S^!$ be the quadratic dual algebra of $S$. Let $\theta_z$ be the Clifford map corresponding to $z$. Then there is an equivalence of triangulated categories $$D^b(\gr_{\mathbb Z_2}E(\theta_z))\cong \underline{\mcm} A.$$
\end{theorem}

Let $\gr A$ be the category of finitely generated right graded $A$-modules , and $\tor A$ be the full subcategory consisting of finite dimensional ones. Set $\qgr A=\gr A/\tor A$ to be the quotient category, which is again an abelian category. If $\qgr A$ has finite global dimension, then $A$ is called {\it a noncommutative isolated singularity} (cf. \cite{Ue}),  or equivalently, Proj$A$ is {\it smooth} (cf. \cite{SvdB}).

Let $S, z, A$ be as above. Assume $\gldim(S)\ge2$. In \cite[Proposition 5.2(2)]{SvdB}, Smith and Van den Bergh showed that if $C(A)$ is semisimple, then Proj$A$ is smooth. Moreover, they also showed that the converse is also true in some special case (cf. \cite[Theorem 5.6]{SvdB}). It is natural to ask whether the converse statement holds true in general. By applying Theorem \ref{thm0.1}, we have the following result, which essentially gives an affirmative answer to this question (cf. Theorem \ref{thm2}).

\begin{theorem}\label{thm0.2} Retain the notation as Theorem \ref{thm0.1}, and assume $\gldim(S)\ge2$.  Then $A$ is a noncommutative isolated singularity if and only if $E(\theta_z)$ is a semisimple $\mathbb Z_2$-graded algebra.
\end{theorem}

Note that the sufficiency part of the above theorem is essentially a consequence of \cite[Proposition 5.2(2)]{SvdB}. For the necessity part, Smith and Van den Bergh have also shown a special case, say if $S$ has Hilbert series $(1-t)^{-4}$ (cf. \cite[Theorem 5.6]{SvdB}). In this case, the corresponding algebra $C(A)$ is of dimension $8$, and their proof depends on a
detailed analysis of the representations of $C(A)$. We mention that the method we use to prove the necessity part of Theorem 0.2 is totally different from that of  \cite[Theorem 5.6]{SvdB}, and our method works for general Koszul Artin-Schelter regular algebras.

Compared to the algebra $C(A)$ constructed in \cite{SvdB}, the Clifford deformations of a Koszul Frobenius algebra are relatively easy to determine. Especially, when $S$ has lower global dimensions, it is possible to find all the quadric hypersurfaces obtained from $S$ and determine whether they are isolated singularities or not, by a detailed analysis on the possible structures of Clifford deformations $E(\theta)$ (cf. Section \ref{sec-exam}).

Another advantage of Clifford deformations is that they could give a new explanation of Kn\"{o}rrer's Periodicity Theorem (cf. \cite[Theorem 3.1]{K}, and noncommutative case \cite[Theorem 1.7]{CKMW}), at least for quadric hypersurfaces. The method used in \cite{K} or in \cite{CKMW} to prove Kn\"{o}rrer's Periodicity Theorem are matrix factorizations (more generally, see \cite{MU}).
In our observation, Kn\"{o}rrer's Periodicity Theorem for quadric hypersurfaces may be explained by Clifford deformations of trivial extensions of Koszul Frobenius algebras (cf. Section \ref{sectr}).

Let $\widetilde{E}:=E\oplus {}_\epsilon E(-1)$ be the trivial extension of $E$ (precisely, see Section \ref{sectr}). Then $\widetilde{E}$ is also a Koszul Frobenius algebra. A Clifford map $\theta$ of $E$ induces a Clifford map $\widetilde{\theta}$ of $\widetilde{E}$. Iterate the trivial extensions, we obtain a Koszul algebra $\widetilde{\widetilde{E}}$ and a Clifford map $\widetilde{\widetilde{\theta}}$ of $\widetilde{\widetilde{E}}$. Assume that the base field $\kk=\mathbb C$. Then we have the following periodicity property (cf. Proposition \ref{propt2}).

\begin{proposition} There is an equivalence of abelian categories $\gr_{\mathbb Z_2} \widetilde{\widetilde{E}}(\,\widetilde{\widetilde{\theta}}\,)\cong \gr_{\mathbb Z_2}E(\theta)$.
\end{proposition}

Retain the notation of Theorem \ref{thm0.1}, and assume $\kk$ is the field of complex numbers. A {\it double branched cover} of $A$ is defined to be the Artin-Schelter Gorenstein algebra $A^\#=S[v]/(z+v^2)$ and the {\it second double branched cover} of $A$ is defined to be the Artin-Schelter Gorenstein algebra $A^{\#\#}=S[v_1,v_2]/(z+v_1^2+v_2^2)$ (cf. \cite[Chapter 8]{LW}, and \cite{CKMW}). The above periodicity property of Clifford deformations of iterate trivial extensions of Koszul Frobenius algebras implies the following Kn\"{o}rrer Periodicity Theorem for noncommutative quadric hypersurfaces (cf. Theorems \ref{thmk} and \ref{thmkp}). We remark that similar results were also obtained in \cite{MU2} recently by using noncommutative matrix factorizations.

\begin{theorem} Retain the notation as above. Assume that $\gldim S\ge2$. Then

{\rm(i)} $A$ is a noncommutative isolated singularity if and only if so is $A^\#$.

{\rm(ii)} there is an equivalence of triangulated categories $\underline{\mcm} A\cong \underline{\mcm} A^{\#\#}$.
\end{theorem}

For a concrete Koszul Artin-Schelter regular algebra $S$, Clifford deformations of the quadratic dual $E=S^!$ is relatively easy to determined. In the last section, we give detailed computations of Clifford deformations for the quadratic dual algebra of a concrete Koszul Artin-Schelter regular algebra $S$ of global dimension 3, and then give a list of all the possible noncommutative quadric hypersurfaces associated to $S$ (up to isomorphism). We also determine whether the obtained noncommutative hypersurfaces are isolated singulariteis or not.

\section{Preliminaries}

Let $\kk$ be a field of characteristic zero. A $\mathbb Z$-graded algebra $A=\oplus_{n\in\mathbb Z}A_n$ is called a {\it connected graded algebra} if $A_n=0$ for $n<0$ and $A_0=\kk$. Let $\Gr A$ denote the category whose objects are right graded $A$-modules, and whose morphisms are right $A$-module morphisms which preserve the gradings of modules. Let $\gr A$ denote the full subcategory of $\Gr A$ consisting of finitely generated graded $A$-modules. For a right graded $A$-module $M$ and an integer $l$, we write $M(l)$ for the right graded $A$-module whose $i$th part is $M(l)_i=M_{i+l}$.

For right graded $A$-modules $M$ and $N$, denote $\uHom_A(M,N)=\bigoplus_{i\in\mathbb Z}\Hom_{\Gr A}(M,N(i))$. Then $\uHom_A(M,N)$ is a $\mathbb Z$-graded vector space. Write $\uExt_A^i$ for the $i$th derived functor of $\uHom_A$. Hence $\uExt_A^i(M,N)$ is also a $\mathbb Z$-graded vector space for each $i\ge0$. We mention that if $A$ is Noetherian and $M$ is finitely generated, then $\uExt_A^i(M,N)\cong \Ext_A^i(M,N)$, the usual extension group in the category of right $A$-modules.

\begin{definition} \cite{AS} A connected graded algebra $A$ is called an {\it Artin-Schelter Gorenstein algebra} of injective dimension $d$ if
\begin{itemize}
  \item [(i)] $A$ has finite injective dimension $\id {}_AA=\id A_A=d<\infty$,
  \item [(ii)] $\uExt_A^i(\kk,A)=0$ for $i\neq d$, and $\uExt_A^d(\kk,A)\cong \kk(l)$,
  \item [(iii)] the left version of (ii) holds.
\end{itemize}
The number $l$ is called the {\it Gorenstein parameter}.
If further, $A$ has finite global dimension, then $A$ is called an {\it Artin-Schelter regular} algebra.
\end{definition}

We need the following lemma, which follows from the Rees-Lemma \cite[Proposition 3.4(b)]{Le}, or the Base-change for the spectral sequence \cite[Exercise 5.6.3]{We}.
\begin{lemma}\label{lem9} Let $A$ be an Artin-Schelter regular algebra of global dimension $d$ with Gorenstein parameter $l$. Let $z\in A_k$ be a central regular element of $A$. Then $A/Az$ is an Artin-Schelter Gorenstein algebra of injective dimension $d-1$ with Gorenstein parameter $l-k$.
\end{lemma}

A locally finite connected graded algebra $A$ is called a {\it Koszul algebra} (cf. \cite{P}) if the trivial module $\kk_A$ has a free resolution $$\cdots\longrightarrow P^{n}\longrightarrow\cdots\longrightarrow P^1\longrightarrow P^0\longrightarrow\kk\longrightarrow0,$$ where $P^n$ is a graded free module generated in degree $n$ for each $n\ge0$. Locally finite means that each $A_i$ is of finite dimension. A Koszul algebra is a quadratic algebra in the sense that $A\cong T(V)/(R)$, where $V$ is a finite dimensional vector space and $R\subseteq V\otimes V$. For a Koszul algebra $A$, the {\it quadratic dual} of $A$ is the quadratic algebra $A^!=T(V^*)/(R^\bot)$, where $V^*$ is the dual vector space and $R^\bot\subseteq V^*\otimes V^*$ is the orthogonal complement of $R$. Note that $A^!$ is also a Koszul algebra.

A $\mathbb Z$-graded finite dimensional algebra $E$ is called a {\it graded Frobenius} algebra if there is an isomorphism of right graded $E$-modules $E\cong E^*(l)$ for some $l\in\mathbb Z$, or equivalently, there is a nondegenerate associative bilinear form $\langle-,-\rangle\colon E\times E\to \kk$ such that for homogeneous elements $a\in E_i,b\in E_j$, $\langle a,b\rangle=0$ if $i+j\neq l$. We will freely use the following connections between graded Frobenius algebras and Koszul Artin-Schelter regular algebras.

\begin{lemma}\cite[Proposition 5.10]{Sm}\label{lempre} Let $A$ be a Koszul algebra, and let $A^!$ be its quadratic dual. Then $A$ is Artin-Schelter regular if and only if $A^!$ is a graded Frobenius algebra.
\end{lemma}

\noindent{\bf Conventions.} In this paper, without otherwise stated, a graded algebra always means a $\mathbb Z$-graded algebra.

\section{Clifford deformation of Koszul Frobenius algebras}

Let $V$ be a finite dimensional vector space and $R\subseteq V\otimes V$. Suppose that $E=T(V)/(R)$ is a Koszul Frobenius algebra.

\begin{definition}\label{def1} {\rm Let $\theta:R\to \kk$ be a linear map. If $(\theta\otimes 1-1\otimes \theta)(V\otimes R\cap R\otimes V)=0$, then we call $\theta$ a {\it Clifford map} of the Koszul Frobenius algebra $E$, and call the associative algebra $$E(\theta)=T(V)/(r-\theta(r):r\in R)$$ the {\it Clifford deformation} of $E$ associated to $\theta$.}
\end{definition}

We view $T(V)$ as a filtered algebra with the filtration: $F_iT(V)=0$ for $i<0$, $F_iT(V)=\sum_{j=0}^i V^{\otimes j}$ for $i\ge0$. The filtration on $T(V)$ induces a filtration
\begin{equation}\label{eq1}
  0\subseteq F_0E(\theta)\subseteq F_1E(\theta)\subseteq\cdots \subseteq F_iE(\theta)\subseteq
\end{equation}
on $E(\theta)$ making $E(\theta)$  a filtered algebra. Let $gr E(\theta)$ be the graded algebra associated to the filtration (\ref{eq1}). The next result is a special case of \cite[Theorem 2.1, Chapter 5]{PP}.

\begin{proposition}\label{prop1} Let $E(\theta)$ be a Clifford deformation of $E$. Then $gr E(\theta)\cong E$ as graded algebras.
\end{proposition}

For later use, we define a linear transformation $\mathfrak{c}$ of $T(V)$ by setting $$\mathfrak{c}(v_1\otimes \cdots \otimes v_n)=\sum_{i=1}^nv_i\otimes\cdots\otimes v_n\otimes  v_1\otimes\cdots\otimes v_{i-1}$$ for $n\ge2$ and $v_1,\dots,v_n\in V$, $\mathfrak{c}(v)=v$ for all $v\in V$, and $\mathfrak{c}(1)=1$.

\begin{example}\label{ex1} Let $V$ be a finite dimensional vector space over the field of real numbers $\mathbb R$ with a basis $\{x_1,\dots,x_n\}$. Consider the exterior algebra $E=\wedge V$. Then $E=T(V)/(R)$, where $R$ is the subspace of $V\otimes V$ spanned by $x_ix_j+x_jx_i$ (for simplicity, we omit the notation $\otimes$) for all $1\leq i,j\leq n$. Then $V\otimes R\cap R\otimes V$ admits a basis: $x_i^3$ $(i=1,\dots, n)$, $\mathfrak{c}(x_i^2x_j)$ for $i\neq j$ $(i,j=1,\dots n)$, $\mathfrak{c}(x_ix_jx_k+x_jx_ix_k)$ for pair-wise different triples $(i,j,k)$ $(i,j,k=1,\dots,n)$. Define a linear map $\theta:R\to \kk$ by setting $\theta(x^2_i)=-1$ for $1\leq i\leq p$; $\theta(x_i^2)=1$ for $p+1\leq i\leq p+q$ where $p+q\leq n$; $\theta(x_i^2)=0$ for $i>p+q$ and $\theta(x_ix_j+x_jx_i)=0$ for $i\neq j$. Then it is easy to see that $\theta$ is a Clifford map of $E$, and $E(\theta)$ is a Clifford deformation of the exterior algebra $\wedge V$. In fact, $E(\theta)$ is isomorphic to the Clifford algebra $\mathbb R^{p,q}$ (cf. \cite[Proposition 15.5]{Po}) defined by the quadratic form $\rho(x)=x_1^2+\cdots+x_p^2-x_{p+1}^2-\cdots-x_{p+q}^2$.
\end{example}

\begin{example}\label{exa} Let $V=\kk x\oplus\kk y$, and let $E=T(V)/(R)$, where $R=\text{span}\{x^2,y^2,xy-yx\}$. Define a map $\theta\colon  R\to \kk$ by setting $\theta(x^2)=a$, $\theta(y^2)=b$ and $\theta(xy-yx)=0$, where $a, b$ are arbitrary elements in $\kk$. Note that $V\otimes R\cap R\otimes V=\text{span}\{x^3,y^3,x^2y-xyx+yx^2,xy^2-yxy+y^2x\}\subseteq V\otimes V\otimes V$. Then it is easy to check that $(\theta\otimes 1-1\otimes \theta)(V\otimes R\cap R\otimes V)=0$, hence $\theta$ is a Clifford map of $E$ and $E(\theta)$ is a Clifford deformation of $E$.
\end{example}

\begin{example} Let $V=\kk x\oplus\kk y$, and let $E=\wedge V$ be the exterior algebra. Then the generating relations of $E$ are $R=\text{span}\{x^2,y^2,xy+yx\}$. Define a map $\theta\colon R\to \kk$ by setting $\theta(x^2)=a$, $\theta(y^2)=b$ and $\theta(xy+yx)=c$, where $a,b,c$ are arbitrary elements in $\kk$. By Example \ref{ex1}, we have $V\otimes R\cap R\otimes V=\text{span}\{x^3,y^3,x^2y+xyx+yx^2,xy^2+yxy+y^2x\}$. It is not hard to check that $\theta$ is a Clifford map of $\wedge V$.
\end{example}

\begin{example}\label{ex3} Let $V=\kk x\oplus\kk y\oplus\kk z$, and let $E=T(V)/(R)$, where $R$ is the subspace of $V\otimes V$ spanned by $xz-zx,yz-zy,x^2-y^2,z^2,xy,yx$. Then $E$ is a Koszul Frobenius algebra. Indeed it is the quadratic dual algebra of the Artin-Schelter regular algebra of type $S_2$ (cf. \cite[Table 3.11, P.183]{AS}). One may check that $\dim (V\otimes R\cap R\otimes V)=10$ and $V\otimes R\cap R\otimes V$ has the following basis \begin{eqnarray*}
                z^3&=&z^3,\\
                yxy&=&yxy,\\
                xyx&=&xyx,\\
                (x^2z-xzx)-(y^2z-yzy)+(zx^2-zy^2)&=&-(xzx-zx^2)+(yzy-zy^2)+(x^2z-y^2z),\\
                (xyz-xzy)+zxy&=&-(xzy-zxy)+xyz,\\
                (x^3-xy^2)-y^2x&=&(x^3-y^2x)-xy^2,\\
                (yxz-yzx)+zyx&=&-(yzx-zyx)+yxz,\\
                xz^2-(zxz-z^2x)&=&z^2x+(xz^2-zxz),\\
                yz^2-(zyz-z^2y)&=&z^2y+(yz^2-zyz),\\
                x^2y+(yx^2-y^3)&=&yx^2+(x^2y-y^3),
                \end{eqnarray*} where the elements on the left hand side are written as elements in $V\otimes R$, and the elements on the right hand side in $R\otimes V$. Define a map $\theta\colon R\to \kk$ by setting $\theta(z^2)=1$, $\theta(xy)=1$, $\theta(yx)=1$, $\theta(x^2-y^2)=1$ and $\theta(\text{others})=0$. It is not hard to check that $\theta$ is a Clifford map of $E$, and hence $E(\theta)$ is a Clifford deformation of $E$.
\end{example}

The next example shows that not every Koszul Frobenius algebra admits a nontrivial Clifford deformation.

\begin{example}\label{ex2} Let $V=\kk x\oplus\kk y$, and let $E=T(V)/(R)$, where $R=\text{span}\{y^2,xy+yx,yx+x^2\}$. Then $E$ is a Koszul Frobenius algebra, which is the quadratic dual algebra of Jordan plane (cf. \cite[Introduction]{AS}). The space $V\otimes R\cap R\otimes V$ admits a basis $$\{y^3,xy^2+yxy+y^2x,2y^2x+xyx+yx^2+yxy+x^2y,xyx+x^3+2y^2x+2yx^2\}.$$ Let $\theta:R\to \kk$ be a Clifford map defined by $\theta(y^2)=a,\theta(xy+yx)=b,\theta(yx+x^2)=c$. Then the equations $(\theta\otimes1-1\otimes\theta)(2y^2x+xyx+yx^2+yxy+x^2y)=2ax-by=0$ and $(\theta\otimes1-1\otimes\theta)(xyx+x^3+2y^2x+2yx^2)=(2a+b)x-2cy=0$ imply that $a=b=c=0$.
\end{example}

Clifford maps of a Koszul Frobenius algebra are corresponding to the central elements of degree 2 of its quadratic dual algebras. Let $E=T(V)/(R)$ be a Koszul Frobenius algebra. Let $E^!=T(V^*)/(R^\bot)$ be the quadratic dual algebra of $E$. We may identify $R^*$ with $E^!_2$.
The next lemma is a special case of \cite[Proposition 4.1, Chapter 5]{PP}.

\begin{lemma} \label{lem6} Retain the notation as above. A linear map $\theta\colon R\to \kk$ is a Clifford map of $E$ if and only if $\theta$, when viewed as an element in $E^!_2$, is a central element of $E^!$.
\end{lemma}

Since $E^!_2\cong R^*$, the above lemma shows that the set of the Clifford maps of $E$ is in one to one correspondence with the set of the central elements in $E_2^!$.

By Lemma \ref{lempre}, every Koszul Frobenius algebra is dual to a Koszul Artin-Schelter regular algebra. Let $S=T(U)/(R)$ be a Koszul Artin-Schelter regular algebra, and let $E:=S^!=T(U^*)/(R^\bot)$. Denote by $\pi_S\colon T(U)\to S$ the natural projection map.

Let $z\in S_2$ be a central element of $S$. Pick an element $r_0\in U\otimes U$ such that $\pi_S(r_0)=z$. Since $R^\bot\subseteq U^*\otimes U^*$, the element $r_0$ defines a map
\begin{equation}\label{eqtheta}
  \theta_z\colon R^\bot\to\kk,\text{ by setting }\theta_z(\alpha)=\alpha(r_0),\ \forall \alpha\in R^\bot.
\end{equation}

Lemma \ref{lem6} implies the following result.
\begin{lemma}\label{lem8} Retain the notation as above. The map $\theta_z\colon R^\bot\to\kk$ is a Clifford map of the Koszul Frobenius algebra $E=T(U^*)/(R^\bot)$.
\end{lemma}

\begin{remark}\label{rem2} Note that the Clifford map $\theta_z$ is independent of the choice of $r_0$. In fact, if $r'_0\in U\otimes U$ is another element such that $\pi_S(r'_0)=z$, then for every element $\alpha\in R^\bot$, one has $\alpha(r_0)=\alpha(r'_0)$. Henceforth, we say that $\theta_z$ is the {\it Clifford map of $E(=S^!)$ corresponding to the central element $z$}.
\end{remark}

\section{Clifford deformations as $\mathbb {Z}_2$-graded Frobenius algebras}

Let $G$ be a group, and let $A=\oplus_{g\in G}A_g$ be a $G$-graded algebra. Set $A^*:=\oplus_{g\in G}(A_{g^{-1}})^*$. Then $A^*$ is a $G$-graded $A$-bimodule, whose degree $g$ component is $(A_{g^{-1}})^*$. Let $M$ be a left $G$-graded $A$-module. For $g\in G$, let $M(g)$ be the left $G$-graded $A$-module whose degree $h$-component $M(g)_h$ is equal to $M_{hg}$.

Similar to $\mathbb Z$-graded Frobenius algebras, a $G$-graded algebra $A$ is called a {\it $G$-graded Frobenius algebra} (cf. \cite{DNN} for instance) if there is an element $g\in G$ and an isomorphism $\varphi\colon A\to A^*(g)$ of left $G$-graded $A$-modules. Equivalently, there is a homogenous bilinear form $\langle\ ,\ \rangle\colon A\times A\longrightarrow\kk(g)$ such that $\langle ab,c\rangle=\langle a,bc\rangle$, where $\kk(g)$ is the $G$-graded vector space concentrated in degree $g^{-1}$.

Let $E=T(V)/(R)$ be a Koszul Frobenius algebra. Let $\theta\colon R\to \kk$ be a Clifford map of $E$. We may view the $\mathbb Z$-graded algebra $T(V)$ as a $\mathbb{Z}_2$-graded algebra by setting $T(V)_0=\kk\oplus \bigoplus_{n\ge1}V^{\otimes 2n}$ and $T(V)_1=\bigoplus_{n\ge1}V^{\otimes 2n-1}$.

Consider the Clifford deformation $E(\theta)=T(V)/(r-\theta(r):r\in R)$. Since $R\subseteq V\otimes V$, it follows that $R_\theta=\text{span}\{r-\theta(r)| r\in R\}$ is a subspace of $T(V)_0$. Hence the ideal $(R_\theta)$ is homogeneous. Therefore we have the following observation.

\begin{lemma} \label{lem1}
Retain the notation as above. The Clifford deformation $E(\theta)$ is a $\mathbb{Z}_2$-graded algebra.
\end{lemma}

\begin{remark}\label{rem1} Note that the filtration (\ref{eq1}) induces a homomorphism $\psi\colon \gr E(\theta)\to E$ of $\mathbb Z$-graded algebras. By Proposition \ref{prop1}, $\psi$ is an isomorphism. By the definition of the $\mathbb Z_2$-grading of $E(\theta)$, we see $\dim E(\theta)_0=\dim (\bigoplus_{i\ge0}E_{2i})$ and $\dim E(\theta)_1=\dim \bigoplus_{i\ge0}E_{2i+1}$.
\end{remark}

Assume that the Frobenius algebra $E$ is of {\it Loevy length} $n$, that is, $E_n\neq 0$ and $E_{i}=0$ for all $i>n$. Since $E$ is $\mathbb Z$-graded Frobenius, $\dim E_n=1$.
Fix a nonzero map $\xi\colon E_n\to \kk$. Let $\phi\colon E(\theta)\to \kk$ be the composition of the following maps
\begin{equation}\label{eqs3}
  \phi\colon E(\theta)\overset{\pi}\to E(\theta)/F_{n-1}E(\theta)\overset{\psi}\to E_n\overset{\xi}\to \kk,
\end{equation}
where $F_{n-1}E(\theta)$ is the $(n-1)$-th part in the filtration (\ref{eq1}), and $\pi$ is the projection. Define a bilinear form $\langle\ ,\ \rangle\colon E(\theta)\times E(\theta)\longrightarrow\kk$ by setting
\begin{equation}\label{eq2}
  \langle a,b\rangle=\phi(ab)
\end{equation}
for all $a,b\in E(\theta)$.

\begin{lemma} \label{lem2} Retain the notation as above. The bilinear form $\langle\ ,\ \rangle$ is nondegenerated, and hence $E(\theta)$ is a Frobenius algebra.
\end{lemma}
\begin{proof} For a nonzero element $a\in E(\theta)$, assume that $a\in F_iE(\theta)$ but $a\in\!\!\!\!\!/ F_{i-1}E(\theta)$. Write $\overline{a}$ for the corresponding element in $F_iE(\theta)/F_{i-1}E(\theta)$. Then $\overline{a}\neq 0$. Since $E$ is a Koszul Frobenius algebra, there is an element $b'\in E_{n-i}$ such that $\psi(\overline{a})b'\neq0$. Let $\overline{b}=\psi^{-1}(b')\in F_{n-i}E(\theta)/F_{n-i-1}E(\theta)$. Pick an  element $b\in F_{n-i}E(\theta)$ which is corresponding to the element $\overline{b}\in F_{n-i}E(\theta)/F_{n-i-1}E(\theta)$. Then $\langle a,b\rangle=\phi(ab)=\xi\psi\pi(ab)=\xi\psi(\overline{a}\overline{b})=\xi(\psi(\overline{a})\psi(\overline{b}))=\xi(\psi(\overline{a})b')\neq0$. Similarly, there is an element $c\in E(\theta)$ such that $\langle c, a\rangle\neq0$. Hence $\langle\ ,\ \rangle$ is nondegenerated.
\end{proof}
\begin{remark} That $E(\theta)$ is selfinjective follows from a more general theory. In fact, $\gr E(\theta)$ is a Frobenius algebra, then by \cite[Chapter I, Lemma 6.11 and Theorem 6.12]{LvO}, $E(\theta)$ is selfinjective. Lemma \ref{lem2} above shows that the bilinear form on $E(\theta)$ indeed inherits from the bilinear form of $E$.
\end{remark}

\begin{proposition}\label{prop4} Retain the notation as above. The bilinear form defined in (\ref{eq2}) is compatible with the $\mathbb Z_2$-grading of $E(\theta)$, and hence $E(\theta)$ is a $\mathbb Z_2$-graded Frobenius algebra.
\end{proposition}
\begin{proof} Assume that the Loevy length of $E$ is $n$. Note that the filtration (\ref{eq1}) is compatible with the $\mathbb Z_2$-grading. Then the map $\phi$ as defined by (\ref{eqs3}) is homogeneous, that is, $\phi\colon E(\theta)\to \kk(g)$, where $g\in \mathbb Z_2$ such that $g=1$ if $n$ is odd, or $g=0$ if $n$ is even. Since the $\langle\ ,\ \rangle$ is the composition of the multiplication of $E(\theta)$ and the map $\phi$, it follows that $\langle\ ,\ \rangle\colon E(\theta)\times E(\theta)\to\kk(g)$ with $g=1$ if $n$ is odd, or $g=0$ if $n$ is even.
\end{proof}

Recall from the definition of the $\mathbb Z_2$-grading on $E(\theta)$ that $E(\theta)_0$ is the quotient space of $\kk\oplus \bigoplus_{k\ge1}V^{\otimes 2k}$. Since the generating relations $R_\theta$ of $E(\theta)$ is concentrated in degree 0 part, we may view $V$ as a subspace of $E(\theta)_1$. Hence each element $a\in E(\theta)_0$ may be written as $a=b+k$ where $b$ is a sum of products of elements in $V$, and $k\in \kk$. Now assume that the Clifford map $\theta$ is non-trivial, and assume that $\theta(r)=k\neq0$ for some $r\in R\subseteq V\otimes V$. Suppose $r=\sum_{i=1}^l x_i\otimes y_i$. Then $\sum_{i=1}^lx_iy_i=k$ in $E(\theta)_0$. Hence $E(\theta)_0=E(\theta)_1E(\theta)_1$. In summary, we have the following result. Recall that a $G$-graded algebra $B=\oplus_{g\in G}B_g$ is said to be {\it strongly graded} \cite{NvO} if $B_gB_h=B_{gh}$ for all $g,h\in G$.

\begin{proposition}\label{prop3} If the Clifford map $\theta$ is nontrivial, then the $\mathbb Z_2$-graded algebra $E(\theta)$ is a strongly graded algebra.
\end{proposition}

\section{Clifford deformations from localizations}

Let $E$ be a Koszul Frobenius algebra. Let $B$ be a graded algebra which is generated in degree 1. Assume that $B$ is a graded extension of $E$ by a central regular element of degree 2, that is, $B$ is a quadratic graded algebra, and there is a central regular element $z\in B_2$ such that $E=B/Bz$. It follows that the degree 1 part of $E$ and $B$ are equal. Since $E$ is a Koszul algebra, $B$ is also a Koszul algebra (cf. \cite[Theorem 1.2]{ST}).

Assume $B=T(V)/(R)$ with $R\subseteq V\otimes V$. Let $\pi:T(V)\to B$ be the projection map. Pick an element $r_0\in V\otimes V$ such that $\pi(r_0)=z$. Since $E=B/Bz$, it follows $E=T(V)/(\kk r_0+R)$. Let $R'=\kk r_0\oplus R$. Define a map
\begin{equation}\label{eq4}
  \theta\colon R'\to \kk, r_0\mapsto 1, r\mapsto 0\text{ for all }r\in R.
\end{equation}

Let us check the elements in $R'\otimes V\cap V\otimes R'$. Assume that $\{r_1,\dots,r_t\}$ is a basis of $R$. For each element $\alpha\in R'\otimes V\cap V\otimes R'$, we have
\begin{equation}\label{eq3}
  \alpha=v_0\otimes r_0+\sum_{i=1}^tv_i\otimes r_i=r_0\otimes v'_0+\sum_{i=1}^tr_i\otimes v'_i,
\end{equation}
where $v_0,\dots,v_t,v'_0,\dots,v'_t\in V$.

\begin{lemma}\label{lem4} In Equation (\ref{eq3}), $v_0=v'_0$.
\end{lemma}
\begin{proof} Assume $v'_0=v_0+u$, for some $u\in V$. By Equation (\ref{eq3}), we have $$v_0\otimes r_0+\sum_{i=1}^tv_i\otimes r_i=r_0\otimes v_0+r_0\otimes u+\sum_{i=1}^tr_i\otimes v'_i.$$
Hence $$r_0\otimes v_0-v_0\otimes r_0+r_0\otimes u=\sum_{i=1}^tv_i\otimes r_i-\sum_{i=1}^tr_i\otimes v'_i.$$ Since the right hand side of the above equation lies in $R\otimes V+V\otimes R$, we have the following identity in the algebra $B$: $$\pi(r_0)\pi(v_0)-\pi(v_0)\pi(r_0)-\pi(r_0)\pi(u)=0.$$ Note that $z=\pi(r_0)$ and $z$ is a central regular element in $B$. It follows that $\pi(u)=0$. Since $\pi$ is injective when it is restricted to $V$, it follows that $u=0$.
\end{proof}

\begin{lemma}\label{lem5} Retain the above notation. The map $\theta$ as defined in (\ref{eq4}) is a Clifford map of $E$.
\end{lemma}
\begin{proof} By Lemma \ref{lem4}, each element $\alpha\in V\otimes R'\cap R'\otimes V$ may be written as $\alpha=v_0\otimes r_0+\sum_{i=1}^tv_i\otimes r_i=r_0\otimes v_0+\sum_{i=1}^tr_i\otimes v'_i$ for some $v_0,\dots,v_t,v'_1,\dots,v'_t\in V$. Then it is easy to see $(\theta\otimes 1-1\otimes \theta)(\alpha)=0$. \end{proof}

Let $\widetilde{R}=\kk(r_0-1)\oplus R$. Then the Clifford deformation of $E$ may be written as $E(\theta)=T(V)/(\widetilde{R})$. Since $B=T(V)/(R)$, we have a natural algebra morphism $f\colon B\to E(\theta)$ such that the diagram
$$\xymatrix{
  T(V) \ar[d]_{\pi} \ar[r]^{\pi_\theta} &  E(\theta)      \\
  B \ar[ur]_{f}                    } $$
commutes, where $\pi_\theta$ is the projection map. Then $f(z)=f(\pi(r_0))=\pi_\theta(r_0)=1$.

Note that $z$ is a central regular element of $B$. Let $B[z^{-1}]$ be the localization of $B$ by the multiplicative set $\{1,z,z^2,\dots\}$. Since $f\colon B\to E(\theta)$ is an algebra morphism such that $f(z)=1$, it induces an algebra morphism $\tilde{f}\colon B[z^{-1}]\to E(\theta)$ such that the following diagram
$$\xymatrix{
  B \ar[d]_{\iota} \ar[r]^{f} &  E(\theta)      \\
  B[z^{-1}] \ar[ur]_{\tilde{f}}                    } $$commutes, where $\iota$ is the inclusion map.
Note that the algebra $B[z^{-1}]$ is a $\mathbb Z$-graded algebra, and $E(\theta)$ is a $\mathbb Z_2$-graded algebra. The next result shows that the degree zero parts $B[z^{-1}]_0$ and $E(\theta)_0$ are isomorphic as algebras, which is motivated by \cite[Lemma 5.1]{SvdB}.

\begin{proposition}\label{prop2} The algebra morphism $\tilde{f}$ induces an isomorphism of algebras $B[z^{-1}]_0\cong E(\theta)_0$.
\end{proposition}
\begin{proof} Since $z$ is of degree 2, it follows that $B[z^{-1}]_0=\sum_{i\ge0}B_{2i}z^{-i}$. Then $$\tilde{f}(B[z^{-1}]_0)=\sum_{i\ge0}f(B_{2i})=E(\theta)_0$$ by the definition of $E(\theta)$ and the hypothesis $E=B/Bz$. Hence $\tilde{f}$ induces an epimorphism $B[z^{-1}]_0\to E(\theta)_0$. By Remark \ref{rem1}, $\dim E(\theta)_0=\dim(\bigoplus_{i\ge0}E_{2i})$. By Lemma \cite[Lemma 5.1(3)]{SvdB} and its proof, we have $\dim B[z^{-1}]_0=\dim(\bigoplus_{i\ge0}E_{2i})=\dim E(\theta)_0$. Hence the restriction of $\tilde{f}$ to $B[z^{-1}]_0$ yields an isomorphism of algebras $B[z^{-1}]_0\cong E(\theta)_0$.
\end{proof}

The structure of $E(\theta)_0$ can be easily determined if the Koszul Frobenius algebra $E$ has lower dimensions (cf. Section \ref{sec-exam}, for detailed computations of $E(\theta)_0$).

\section{Noncommutative quadric hypersurfaces}

Let $A$ be a noetherian Artin-Schelter Gorenstein algebra. Let $M$ be a right graded $A$-module. An element $m\in M$ is called a torsion element, if $mA_{\ge n}=0$ for some $n\ge0$. Let $\Gamma(M)$ be the submodule of $M$ consisting of all the torsion elements. Since $A$ is noetherian, we obtain a functor $\Gamma:\gr A\longrightarrow\gr A$. It is easy to see $\Gamma\cong \underrightarrow{\lim}\uHom_A(A/A_{\ge n},-)$. The $i$th right derived functor of $\Gamma$ is written as $R^i\Gamma$.

For a finitely generated right graded $A$-module $M$, {\it depth} of $M$ is defined to be the number $$\depth(M)=\min\{i|R^i\Gamma(M)\neq 0\}.$$ Assume that $\id {}_AA=\id A_A=d$. Then $M$ is called a {\it maximal Cohen-Macaulay} (or {\it Gorenstein projective}) module if $\depth(M)=d$. Let $\mcm A$ be the subcategory of $\gr A$ consisting of all the maximal Cohen-Macaulay modules. The additive category $\mcm A$ is a Frobenius category with enough projectives and injectives. Let $\underline{\mcm} A$ be the stable category of $\mcm A$. Then $\underline{\mcm} A$ is a triangulated category.

Now let $S=T(V)/(R)$ be a Koszul Artin-Schelter regular algebra. Let $z\in S_2$ be a central regular element of $S$. The quotient algebra $A=S/Sz$ is usually called a (noncommutative) {\it quadric hypersurface}.

Let $\pi_S\colon T(V)\to S$ be the natural projection map. Pick an element $r_0\in V\otimes V$ such that $\pi_S(r_0)=z$. Denote the quadratic dual algebra of $S$ by $E=T(V^*)/(R^\bot)$. Then $E$ is a Koszul Frobenius algebra (cf. Lemma \ref{lempre}).
Since $A=S/Sz$, it follows $A\cong T(V)/(\kk r_0+R)$. As before, write $R'$ for the space $\kk r_0+R$. Note that $\kk r_0\cap R=0$. Then we have $V\otimes V=\kk r_0\oplus R\oplus R''$ for some subspace $R''\subseteq V\otimes V$. Define a linear map $$r^*_0\colon V\otimes V\to \kk$$ by setting $r^*_0(r_0)=1$ and $r_0^*(R)=r_0^*(R'')=0$. We view $r_0^*$ as an element of $V^*\otimes V^*$.

Consider the quadratic dual algebra $A^!=T(V^*)/({R'}^\bot)$. Note that ${R'}^\bot=(\kk r_0)^\bot\cap R^\bot$. On the other hand, ${R'}^\bot+\kk r_0^*=(\kk r_0)^\bot\cap R^\bot+\kk r_0^*=R^\bot$.  Let $\pi_{A^!}:T(V^*)\to A^!$ be the projection map. Then $$w:=\pi_{A^!}(r^*_0)\neq 0.$$
We next check that $w$ is a central element of $A^!$. Note that $A^!_3\cong (R'\otimes V\cap V\otimes R')^*$ as vector spaces. For every $\alpha\in R'\otimes V\cap V\otimes R'$, by Lemma \ref{lem4}, $\alpha=v_0\otimes r_0+\sum_{i=1}^tv_i\otimes r_i=r_0\otimes v_0+\sum_{i=1}^tr_i\otimes v'_i,$
for some $v_0,\dots,v_t,v'_0,\dots,v'_t\in V$, where $r_1,\dots r_t$ is a basis of $R$. Taking an element $f\in V^*\cong A^!_1$, we have $(fw)(\alpha)=f(v_0)r^*_0(r_0)=f(v_0)$, and $(wf)(\alpha)=r^*_0(r_0)f(v_0)=f(v_0)$. Hence $fw=wf$, for all $f\in V^*$. Therefore $w$ is a central element.
Now, it follows $$E\cong A^!/A^!w.$$ Moreover, $w$ is also a regular element of $A^!$ (cf. \cite[Lemma 5.1(2)]{SvdB}).

For a $\mathbb Z_2$-graded algebra $B$, we write $\Gr_{\mathbb Z_2}B$ (resp. $\gr_{\mathbb Z_2}B$) for the category of right $\mathbb Z_2$-graded modules (resp. finitely generated $\mathbb Z_2$-graded modules), whose morphisms are degree 0 right $\mathbb Z_2$-graded morphisms. We arrive at our main observation of the paper, which is a slight improvement of \cite[Proposition 5.2]{SvdB}.

\begin{theorem}\label{thm1} Let $S=T(V)/(R)$ be a noetherian Koszul Artin-Schelter regular algebra, and let $E:=T(V^*)/(R^\bot)$ be the quadratic dual algebra of $S$. Assume that $z\in S_2$ is a central regular element of $S$. Let $\theta_z: R^\bot\to \kk$ be the map as defined in (\ref{eqtheta}). We have the following statements.
\begin{itemize}
  \item [(i)] $\theta_z$ is a Clifford map of $E$, and hence $E(\theta_z)$ is a Clifford deformation of $E$.
  \item [(ii)] \cite[Lemma 5.1(1)]{SvdB} Let $A=S/Sz$. Then $A$ is a Koszul Artin-Schelter Gorenstien algebra.
  \item[(iii)] There is an equivalence of triangulated categories $$\underline{\mcm} A\cong D^b(\gr_{\mathbb Z_2} E(\theta_z)).$$
\end{itemize}
\end{theorem}
\begin{proof} (i) follows from Lemma \ref{lem8}.

(ii) By Lemma \ref{lem9}, $A$ is Artin-Schelter Gorenstein. Since $S$ is a Koszul algebra and $z\in S_2$ is a central regular element, $A$ is a Koszul algebra (cf. \cite[Lemma 5.1(1)]{SvdB}).

(iii) Retain the notation above the theorem. The element $w=\pi_{A^!}(r_0^*)$ is a central regular element of $A^!$. Let $A^![w^{-1}]$ be the localization of $A^!$ at $w$. Then $A^![w^{-1}]$ is a $\mathbb Z$-graded algebra. By \cite[Proposition 5.2]{SvdB}, there is an equivalence of triangulated categories $D^b(\mmod A^![w^{-1}]_0)\cong \underline{\mcm} A$.

Note that ${R'}^\bot=(\kk r_0)^\bot\cap R^\bot$ and $R^\bot={R'}^\bot+\kk r_0^*$. Then $\theta_z({R'}^\bot)=0$ and $\theta_z(r_0^*)=1$. By Proposition \ref{prop2},
\begin{equation}\label{eq6}
 A^![w^{-1}]_0\cong E(\theta_z)_0.
\end{equation}
Hence we have $$D^b(\mmod A^![w^{-1}]_0)\cong D^b(\mmod E(\theta_z)_0).$$ By Proposition \ref{prop3}, $E(\theta_z)$ is a strongly $\mathbb Z_2$-graded algebra. It follows that there is an equivalence of abelian categories $\gr_{\mathbb Z_2}E(\theta_z)\cong \mmod E(\theta_z)_0$ (cf. \cite[Theorem 3.1.1]{NvO}). Therefore, $\underline{\mcm} A\cong D^b(\gr_{\mathbb Z_2} E(\theta_z))$ as triangulated categories.
\end{proof}

\section{Noncommutative quadrics with isolated singularities}

Let $A$ be a noetherian connected graded algebra. Let $\tor A$ be the full subcategory of $\gr A$ consisting of finite dimensional right graded $A$-modules. The quotient category $$\qgr A=\gr A/\tor A$$ is the noncommutative analogue of projective schemes (cf. \cite{AZ,Ve}). For $M\in \gr A$, we write $\mathcal{M}$ for the corresponding object in $\qgr A$. Recall from \cite{Ue} that $A$ is called a {\it noncommutative isolated singularity} if $\qgr A$ has finite global dimension, that is, there is an integer $p$ such that for any objects $\mathcal{M},\mathcal{N}\in\qgr A$, $\Ext^i_{\qgr A}(\mathcal{M},\mathcal{N})=0$ for all $i>p$, or equivalently, the noncommutative projective scheme Proj$A$ is {\it smooth} (cf. \cite{SvdB}).

Let $S$ be a noetherian Koszul Artin-Schelter regular algebra. Assume that $z\in S_2$ is a central regular element of $S$. In this section, we investigate when $A=S/Sz$ is a noncommutative isolated singularity.

Let $per A$ be the triangulated subcategory of $D^b(\gr A)$ consisting of bounded complexes of finitely generated right graded projective $A$-modules. Then we have a quotient triangulated category $D^{gr}_{sg}(A)=D^b(\gr A)/per A$.

Since $A$ is Artin-Schelter Gorenstein, there is an equivalence of triangulated categories \cite[Theorem 4.4.1(2)]{Bu}:
\begin{equation}\label{eq5}
  D^{gr}_{sg}(A)\cong \underline{\mcm} A.
\end{equation}

The triangulated category $D^{gr}_{sg}(A)$ is related to $D^b(\qgr A)$ by Orlov's famous decomposition theorem.

\begin{theorem}\label{or} \cite[Theorem 2.5]{Or} Let $A$ be a noetherian Artin-Schelter Gorenstein algebra of Gorenstein parameter $l$. Then
\begin{itemize}
  \item [(i)] if $l>0$, there are fully faithful functors $\Phi_i: D^{gr}_{sg}(A)\longrightarrow D^b(\qgr A)$ and semiorthogonal decompositions $$D^b(\qgr A)=\langle \pi A(-i-l+1),\dots,\pi A(-i),\Phi_i(D^{gr}_{sg}(A))\rangle,$$ where $\pi:\gr A\to \qgr A$ is the projection functor;
  \item [(ii)] if $l<0$, there are fully faithful functors $\Psi_i:D^b(\qgr A)\to D^{gr}_{sg}(A)$ and semiorthogonal decompositions
  $$D^{gr}_{sg}(A)=\langle q\kk(-i),\dots,q\kk(-i+l+1),\Psi_i(D^b(\qgr A))\rangle,$$ where $q:D^b(\gr A)\to D^{gr}_{sg}(A)$ is the natural projection functor;
    \item [(iii)] if $l=0$, there is an equivalence $$D^{gr}_{sg}(A)\cong D^b(\qgr A).$$
\end{itemize}
\end{theorem}

We have the following special case of Orlov's theorem.
\begin{lemma}\label{lem10} Let $S$ be a noetherian Koszul Artin-Schelter regular algebra of global dimension $d\ge2$, and let $z\in S_2$ be a central regular element of $S$. Set $A=S/Sz$. Then there is a fully faithful triangle functor $\Phi:\underline{\mcm} A\longrightarrow D^b(\qgr A)$.
\end{lemma}
\begin{proof} Since $S$ is a Koszul Artin-Schelter regular algebra of global dimension $d$, then the Gorenstein parameter of $S$ is equal to $d$. By Lemma \ref{lem9}, $A$ is Artin-Schelter Gorenstein of injective dimension $d-1$ with Gorenstein parameter $d-2\ge0$. The lemma follows from the equivalence (\ref{eq5}) and Theorem \ref{or}(i,iii).
\end{proof}

Now let $S=T(V)/(R)$ be a noetherian Koszul Artin-Schelter regular algebra of global dimension $d\ge2$, and  $E=T(V^*)/(R^\bot)$ its quadratic dual. Assume that $z\in S_2$ is a central regular element of $S$. Let $E(\theta_z)$ be the Clifford deformation of $E$ corresponding to the central element $z$ (cf. Theorem \ref{thm1}).

\begin{theorem}\label{thm2} Retain the notation as above. Then $A=S/Sz$ is a noncommutative isolated singularity if and only if $E(\theta_z)$ is semisimple as a $\mathbb Z_2$-graded algebra.
\end{theorem}
\begin{proof} Assume $E(\theta_z)$ is a $\mathbb Z_2$-graded semisimple algebra. By Proposition \ref{prop3}, $E(\theta_z)$ is a strongly $\mathbb Z_2$-graded algebra. Then $\gr_{\mathbb Z_2} E(\theta_z)\cong \mmod E(\theta_z)_0$ as abelian categories. Hence $E(\theta_z)_0$ is a semisimple algebra. By isomorphism (\ref{eq6}) as in the proof of Theorem \ref{thm1}, $E(\theta_z)_0\cong A^![w^{-1}]_0$. Hence $A^![w^{-1}]_0$ is semisimple. By \cite[Proposition 5.2(2)]{SvdB}, $A$ is a noncommutative isolated singularity.

Conversely, assume $A$ is a noncommutative isolated singularity. Then $\qgr A$ has finite global dimension. Given objects $X,Y\in D^b(\qgr A)$, there is an integer $p$ (depending on $X$ and $Y$) such that $\Hom_{D^b(\qgr A)}(X,Y[i])=0$ for $i>p$. Let $J$ be the $\mathbb Z_2$-graded Jacobson radical of $E(\theta_z)$. Write $T$ for the quotient algebra $E(\theta_z)/J$.
By Theorem \ref{thm1}, there is an equivalence of triangulated categories $\Psi:D^b(\gr_{\mathbb Z_2}E(\theta_z))\longrightarrow \underline{\mcm} A$. Let $\Phi:\underline{\mcm} A\longrightarrow D^b(\qgr A)$ be the fully faithful functor in Lemma \ref{lem10}. Then there is an integer $q$ such that for $i>q$, $$\Ext^i_{\gr_{\mathbb Z_2}E(\theta_z)}(T,T)\cong\Hom_{D^b(\gr_{\mathbb Z_2}E(\theta_z))}(T,T[i])\cong\Hom_{D^b(\qgr A)}(\Phi\Psi(T),\Phi\Psi(T)[i])=0.$$
Since $E(\theta_z)$ is finite dimensional, as a right $\mathbb Z_2$-graded $E(\theta_z)$-module $T$ is semisimple and each simple right $\mathbb Z_2$-graded $E(\theta_z)$-module is a direct summand of $T$. It follows that the right $\mathbb Z_2$-graded $E(\theta_z)$-module $T$ has finite projective dimension in $\gr_{\mathbb Z_2} E(\theta_z)$. Hence the $\mathbb Z_2$-graded algebra $E(\theta_z)$ has finite global dimension. On the other hand, $E(\theta_z)$ is a $\mathbb Z_2$-graded Frobenius algebra by Proposition \ref{prop4}. It follows that $E(\theta_z)$ is semisimple as a $\mathbb Z_2$-graded algebra.
\end{proof}

\begin{remark} (i) The sufficiency part of Theorem \ref{thm2} mainly follows from \cite[Proposition 5.2]{SvdB}. Observing that $E(\theta_z)$ is strongly $\mathbb Z_2$-graded, we obtain an abstract proof of \cite[Theorem 5.6]{SvdB}.

(ii) Independently, Mori-Ueyama also gave another proof of the above theorem in \cite[Theorem 5.4]{MU2}.
\end{remark}

\section{Clifford deformations of trivial extensions of Koszul Frobenius algebras}\label{sectr}

In this section, we work over the field of complex numbers $\mathbb{C}$. Let $\mathbb M_2(\mathbb C)$ be the matrix algebra of all the $2\times 2$-matrices over $\mathbb{C}$. We may view $\mathbb M_2(\mathbb{C})$ as a $\mathbb Z_2$-graded algebra by setting the degree 0 part consisting of elements of the form
$\left(
\begin{array}{cc}
     a&0\\
     0&b
\end{array}
\right)
$, and degree 1 part consisting of elements of the form
$\left(
\begin{array}{cc}
     0&a\\
     b&0
\end{array}
\right)$.

Let $A$ and $B$ be $\mathbb Z_2$-graded algebras. The twisting tensor algebra $A\hat\otimes B$ is a $\mathbb Z_2$-graded algebra defined as follows: $(A\hat\otimes B)_0=A_0\otimes B_0\oplus A_1\otimes B_1$ and $(A\hat\otimes B)_1=A_0\otimes B_1\oplus A_1\otimes B_0$ as $\mathbb Z_2$-graded space; the multiplication is defined as $(a_1\hat\otimes b_1)(a_2\hat\otimes b_2)=(-1)^{|b_1||a_2|}(a_1a_2\hat \otimes b_1b_2)$, where $b_1$ and $a_2$ are homogeneous elements with degrees $|b_1|$ and $|a_2|$ respectively.

Let $\mathbb G=\{1,\alpha\}$ be a group of order 2, and let $\mathbb C \mathbb G$ be the group algebra. Then $\mathbb C \mathbb G$ is naturally a $\mathbb Z_2$-graded algebra by setting $|\alpha|=1$ and $|1|=0$. Define a linear map
\begin{equation}\label{TE1}
\Upsilon:\mathbb C \mathbb G\hat\otimes\mathbb C \mathbb G\longrightarrow\mathbb M_2(\mathbb C),
\end{equation}by setting
{\small\begin{equation*}
1\hat\otimes1\mapsto \left(
\begin{array}{cc}
     1&0\\
     0&1
\end{array}
\right),\
1\hat\otimes \alpha\mapsto \left(
\begin{array}{cc}
     0&1\\
     1&0
\end{array}
\right),\
\alpha\hat\otimes 1\mapsto\left(
\begin{array}{cc}
     0&\sqrt{-1}\\
     -\sqrt{-1}&0
\end{array}
\right),\ \alpha\hat\otimes\alpha\mapsto\left(
\begin{array}{cc}
     \sqrt{-1}&0\\
     0&-\sqrt{-1}
\end{array}
\right).
\end{equation*}}
The following lemma is well known and is easy to check.

\begin{lemma} The map $\Upsilon$ is an isomorphism of $\mathbb Z_2$-graded algebras.
\end{lemma}

Note that there is an equivalence of abelian categories $$\gr_{\mathbb Z_2} \mathbb M_2(\mathbb C)\cong \gr_{\mathbb Z_2}\mathbb C,$$ where $\mathbb C$ is viewed as a $\mathbb Z_2$-graded algebra concentrated in degree 0. By the above lemma, we have an equivalence of abelian categories (see also \cite[Lemma 4.11]{Z}) $$\gr_{\mathbb Z_2} (\mathbb C\mathbb G\hat\otimes\mathbb C\mathbb G)\cong \gr_{\mathbb Z_2}\mathbb C.$$ We have the following result, which should be well known for experts.

\begin{lemma}\label{lemTE1} Let $A$ be a $\mathbb Z_2$-graded algebra. Then there is an equivalence of abelian categories
$$\gr_{\mathbb Z_2}(A\hat\otimes\mathbb C\mathbb G\hat\otimes\mathbb C\mathbb G)\cong \gr_{\mathbb Z_2} A.$$
\end{lemma}
\begin{proof} This is a direct consequence of \cite[Lemma 3.10]{Z}. Note that the algebra is assumed to be finite dimensional in \cite[Lemma 3.10]{Z}, but the result hold for arbitrary $\mathbb Z_2$-graded algebras.
\end{proof}
\begin{lemma}\label{lemt3} For $P\in \gr_{\mathbb Z_2} A\hat\otimes\mathbb C\mathbb G$, $P$ is projective if and only if $P$ is projective as a $\mathbb Z_2$-graded $A$-module.
\end{lemma}
\begin{proof} Note that $A\hat\otimes\mathbb C\mathbb G$ is a $\mathbb Z_2$-graded projective $A$-module. Hence each $\mathbb Z_2$-graded projective $A\hat\otimes\mathbb C\mathbb G$-module is projective as a $\mathbb Z_2$-graded $A$-module.

On the contrary, suppose that $P$ is projective as a $\mathbb Z_2$-graded $A$-module. Let $f\colon M\to N$ be an epimorphism of $\mathbb Z_2$-graded $A\hat\otimes\mathbb C\mathbb G$-modules, and let $g\colon P\to N$ be a $\mathbb Z_2$-graded $A\hat\otimes\mathbb C\mathbb G$-module morphism. Since $P$ is projective as a graded $A$-module, there is a graded $A$-module morphism $h\colon P\to M$ such that $fh=g$. Define a morphism $h'\colon P\to M$ by $h'(p)=\frac{1}{2}(h(p)+h(p\cdot \alpha)\cdot \alpha)$, for all $p\in P$. It is straightforward to check that $h'$ is a graded $A\hat\otimes\mathbb C\mathbb G$-module morphism, and that $fh'=g$. Hence $P$ is projective in $\gr_{\mathbb Z_2} A\hat\otimes\mathbb C\mathbb G$.
\end{proof}

Let $\gldim_{\mathbb Z_2} A$ be the $\mathbb Z_2$-graded global dimension of $A$.

\begin{corollary}\label{cort1} $\gldim_{\mathbb Z_2} A=\gldim_{\mathbb Z_2} A\hat\otimes\mathbb C\mathbb G$.
\end{corollary}
\begin{proof} Let $M$ be a right $\mathbb Z_2$-graded $A\hat\otimes\mathbb C\mathbb G$-module. Assume that $\gldim_{\mathbb Z_2} A=d<\infty$. Suppose that $$\cdots\to P_n\overset{\delta_n}\to\cdots\overset{\delta_1}\to P_0\overset{\delta_0}\to M\to 0$$ is a projective resolution of $M$ in $\gr_{\mathbb Z_2} A\hat\otimes\mathbb C\mathbb G$. Since $\gldim_{\mathbb Z_2} A=d$, $Q=\ker\delta_{d-1}$ is projective as a graded $A$-module, and hence $Q$ is projective graded $A\hat\otimes\mathbb C\mathbb G$-module by Lemma \ref{lemt3}. Therefore, $\gldim_{\mathbb Z_2} A\hat\otimes\mathbb C\mathbb G\leq \gldim_{\mathbb Z_2}A$. Similarly, $\gldim_{\mathbb Z_2} A\hat\otimes\mathbb C\mathbb G\hat\otimes\mathbb C\mathbb G\leq\gldim_{\mathbb Z_2} A\hat\otimes\mathbb C\mathbb G$. By Lemma \ref{lemTE1}, $\gldim_{\mathbb Z_2} A=\gldim_{\mathbb Z_2} A\hat\otimes\mathbb C\mathbb G\hat\otimes\mathbb C\mathbb G$.

If $\gldim_{\mathbb Z_2} A=\infty$, we claim that $\gldim_{\mathbb Z_2} A\hat\otimes\mathbb C\mathbb G=\infty$. Otherwise, if $\gldim_{\mathbb Z_2} A\hat\otimes\mathbb C\mathbb G<\infty$, then by the above proof, it follows that $\gldim_{\mathbb Z_2} A\hat\otimes\mathbb C\mathbb G\hat\otimes\mathbb C\mathbb G\leq\gldim_{\mathbb Z_2} A\hat\otimes\mathbb C\mathbb G$. Then by Lemma \ref{lemTE1}, $\gldim_{\mathbb Z_2} A<\infty$, a contradiction. Hence the result follows.
\end{proof}

Now let $E=T(V)/(R)$ be a Koszul Frobenius algebra, and let $M$ be a graded $E$-bimodule. Recall that the trivial extension of $E$ by $M$ is the $\mathbb Z$-graded algebra $\Gamma(E,M)=E\oplus M$ with a multiplication defined by $(x_1,m)(x_2,n)=(x_1x_2,x_1n+mx_2)$ for $x_1,x_2\in E$ and $m,n\in M$.

Let $\epsilon$ be the automorphism of $E$ defined by $\epsilon(x)=(-1)^{|x|}x$ for homogeneous element $x\in E$. Let ${}_\epsilon E$ be the graded $E$-bimodule with left $E$-action twisted by $\epsilon$.
Denote
\begin{equation}\label{eqt}
  \widetilde{E}=\Gamma(E,{}_\epsilon E(-1)).
\end{equation}

\begin{lemma}\label{lemt1} Let $E$ be a Koszul Frobenius algebra. Then $\widetilde{E}$ is a Koszul Frobenius algebra.
\end{lemma}
\begin{proof} Let $S=E^!=T(V^*)/(R^\bot)$ be the quadratic dual of $E$. Then $S$ is a Koszul Artin-Schelter regular algebra by Lemma \ref{lempre}. Let $S^\natural=S[\alpha]$ be the polynomial algebra with coefficients in $S$. It is well known that $S^\natural$ is a Koszul Artin-Schelter regular algebra. Let $V^\natural=V^*\oplus \kk\alpha$ and $$R^\natural=R^\bot\oplus \text{span}\{\beta\otimes \alpha-\alpha\otimes\beta|\beta\in V^*\}.$$ The Koszul algebra $S^\natural$ may be written as $S^\natural=T(V^\natural)/(R^\natural)$.

By \cite[Proposition 2.2]{HVZ}, it follows that $\widetilde{E}$ is the quadratic dual algebra of $S^\natural$. By Lemma \ref{lempre}, we obtain that $\widetilde{E}$ is Koszul Frobenius.
\end{proof}

We may write the trivial extension $\widetilde{E}$ by generators and relations. Let $\widetilde{V}=V\oplus \kk y$ and let
\begin{equation}\label{eqt2}
  \widetilde{R}=R\oplus R'\oplus \kk y\otimes y,\text{ where }R'=\text{span}\{x\otimes y+y\otimes x|x\in V\}.
\end{equation} Then $\widetilde{E}=T(\widetilde{V})/(\widetilde{R})$.

Assume that $\theta:R\to \kk$ is a Clifford map of $E$. Define a linear map
\begin{equation}\label{eqt3}
  \widetilde{\theta}:\widetilde{R}\to\kk,
\end{equation}
\begin{equation}\label{eqt4}
  \widetilde{\theta}(r)=\theta(r),\text{ for }r\in R;\ \widetilde{\theta}(y\otimes y)=1;\ \widetilde{\theta}(x\otimes y+y\otimes x)=0,
\text{ for }x\in V.
\end{equation}

\begin{lemma}\label{lemt2} Retain the notation as above. $\widetilde{\theta}$ is a Clifford map of the Koszul Frobenius algebra $\widetilde{E}$, and hence $\widetilde{E}(\widetilde{\theta})$ is a Clifford deformation of $\widetilde{E}$.
\end{lemma}
\begin{proof} Let $\{x_1,\dots,x_n\}$ be a basis of $V$. Then $\{x_i\otimes y+y\otimes x_i|i=1,\dots,n\}$ is a basis of $R'$. We may write $$\widetilde{V}\otimes\widetilde{R}=(V\otimes R)\oplus (V\otimes R')\oplus (V\otimes y\otimes y)\oplus (y\otimes R)\oplus (y\otimes R')\oplus(\kk y\otimes y\otimes y),$$ and similarly, $$\widetilde{R}\otimes\widetilde{V}=(R\otimes V)\oplus (R'\otimes V)\oplus (y\otimes y\otimes V)\oplus (R\otimes y)\oplus (R'\otimes y)\oplus(\kk y\otimes y\otimes y).$$ For an element $w\in\widetilde{V}\otimes\widetilde{R}$, we may write $$w=w_1+w_2+w_3+w_4+w_5+k y\otimes y\otimes y$$ for $w_1\in V\otimes R$, $w_2\in V\otimes R'$, $w_3\in V\otimes y\otimes y$, $w_4\in y\otimes R$, and $w_5\in y\otimes R'$.
Now assume $w$ is also in $\widetilde{R}\otimes\widetilde{V}$. By comparing the multiplicity of the element $z$ in the tensor products, we see $$w_1\in R\otimes V,\ w_2+w_4\in R'\otimes V\oplus R\otimes y,\ w_3+w_5\in y\otimes y\otimes V\oplus  R'\otimes y.$$ Assume $w_2=\sum_{i,j=1}^na_{ij}x_i\otimes(x_j\otimes y-y\otimes x_j)$, $w_4=y\otimes\sum_{i,j=1}^nb_{ij}x_i\otimes x_j$. Then
\begin{eqnarray*}
w_2+w_4&=&\sum_{i,j=1}^na_{ij}x_i\otimes(x_j\otimes y-y\otimes x_j)+y\otimes\sum_{i,j=1}^nb_{ij}x_i\otimes x_j\\
&=&\sum_{i,j=1}^na_{ij}(x_i\otimes x_j)\otimes y+\sum_{j=1}^n\sum_{i=1}^n(-a_{ij}x_i\otimes y+b_{ij}y\otimes x_i)\otimes x_j.
\end{eqnarray*}
Since $w_2+w_4\in R'\otimes V\oplus R\otimes z$, we have $\sum_{j=1}^n\sum_{i=1}^n(-a_{ij}x_i\otimes y+b_{ij}y\otimes x_i)\otimes x_j\in R'\otimes V$. Then $a_{ij}=b_{ij}$ for all $i,j$, and it follows that $$(1\otimes\widetilde{\theta})(w_2+w_4)=\theta(\sum_{i,j=1}^na_{ij}x_i\otimes x_j)y=(\widetilde{\theta}\otimes 1)(w_2+w_4).$$ Similarly, we have $(1\otimes\widetilde{\theta})(w_3+w_5)=(\widetilde{\theta}\otimes 1)(w_3+w_5)$. Since $\theta$ is a Clifford map of $E$, $(1\otimes \theta)(w_1)=(\theta\otimes 1)(w_1)$. Therefore, $(1\otimes\widetilde{\theta})(w)=(\widetilde{\theta}\otimes 1)(w)$.
\end{proof}

\begin{proposition}\label{propt1} Retain the notation as above. There is an isomorphism of $\mathbb Z_2$-graded algebras $\widetilde{E}(\widetilde{\theta})\cong E(\theta)\hat\otimes \mathbb C \mathbb G$.
\end{proposition}
\begin{proof} Note that $E(\theta)=T(V)/(r-\theta(r):r\in R)$ and $\widetilde{E}(\widetilde{\theta})=T(\widetilde{V})/(r-\widetilde{\theta}(r):r\in\widetilde{R})$. Let $\pi:T(V)\to E(\theta)$ be the projection map. We may define a linear map $\psi:T(\widetilde{V})\longrightarrow E(\theta)\hat\otimes \mathbb C \mathbb G$ by setting $\psi(x)=\pi(x)\hat\otimes 1$ for $x\in V$ and $\psi(y)=1\hat\otimes\alpha$. The map $\psi$ induces an algebra epimorphism $\overline{\psi}:\widetilde{E}(\widetilde{\theta})\longrightarrow E(\theta)\hat\otimes \mathbb C \mathbb G$. Clearly, $\overline{\psi}$ preserves the $\mathbb Z_2$-grading. Since $E(\theta)\hat\otimes \mathbb C \mathbb G$ and $\widetilde{E}(\widetilde{\theta})$ have the same dimension as vector spaces, $\overline{\psi}$ is an isomorphism.
\end{proof}

As a special case of Corollary \ref{cort1}, we have the following result.

\begin{corollary}\label{cort2} $\gldim_{\mathbb Z_2} \widetilde{E}(\widetilde{\theta})=\gldim_{\mathbb Z_2} E(\theta)$.
\end{corollary}

We may iterate the above construction. By Lemma \ref{lemt1}, $\widetilde{E}$ is a Koszul Frobenius algebra, and the trivial extension $\widetilde{\widetilde{E}}$ of $\widetilde{E}$ is again a Koszul Frobenius algebra. The Clifford map $\widetilde{\theta}$ of the Koszul algebra $\widetilde{E}$ may be extended to a Clifford map $\widetilde{\widetilde{\theta}}$ of $\widetilde{\widetilde{E}}$ by the way as in (\ref{eqt3}) and (\ref{eqt4}). Hence we obtain a $\mathbb Z_2$-graded algebra $\widetilde{\widetilde{E}}(\,\widetilde{\widetilde{\theta}}\,)$.

\begin{proposition}\label{propt2} There is an equivalence of abelian categories $\gr_{\mathbb Z_2} \widetilde{\widetilde{E}}(\,\widetilde{\widetilde{\theta}}\,)\cong \gr_{\mathbb Z_2}E(\theta)$.
\end{proposition}
\begin{proof} By Proposition \ref{propt1}, $\widetilde{\widetilde{E}}(\,\widetilde{\widetilde{\theta}}\,)\cong\widetilde{E}(\widetilde{\theta})\hat\otimes\mathbb C \mathbb G\cong E(\theta)\hat\otimes \mathbb C \mathbb G\hat\otimes\mathbb C \mathbb G$. Now the result follows from Lemma \ref{lemTE1}.
\end{proof}

\section{Kn\"{o}rrer Periodicity for noncommutative quadrics revisited}

In this section, the base field is assumed to be $\kk=\mathbb C$. Let $S=T(V)/(R)$ be a noetherian Koszul Artin-Schelter regular algebra of $\gldim A\ge2$. Let $z\in S_2$ be a central regular element of $S$, and set $A=S/(z)$. The {\it double branched cover} of $A$ is defined to be the algebra (cf. \cite{K,LW,CKMW}) $$A^\#:=S[v_1]/(z+v_1^2).$$
We write $S[v_1]=T(U)/(R')$, where $U=V\oplus \kk v_1$ and $R'=R\oplus\{v\otimes v_1-v_1\otimes v|v\in V\}.$

Let $E=S^!$ be the quadratic dual algebra of $S$, and $\widetilde{E}$ the trivial extension of $E$ as defined by (\ref{eqt}).
Then by \cite[Proposition 2.2]{HVZ}, $\widetilde{E}$ is isomorphic to the quadratic dual algebra $S[v_1]^!$.

Denote by $\pi_S\colon T(V)\to S$ the natural projection map. Pick an element $r_0\in V\otimes V$ such that $\pi_S(r_0)=z$. Let $\theta_z$ be the Clifford map as defined by (\ref{eqtheta}) (cf. Lemma \ref{lem8}), and $E(\theta_z)$ the Clifford deformation of $E$ associated to $\theta_z$.

Let $\pi_{S[v_1]}\colon T(U)\to S[v_1]$ be the projection map. Set $r_0^\#=r_0+v_1\otimes v_1$. Then $\pi_{S[v_1]}(r_0^\#)=z+v_1^2$. Let $\widetilde{\theta_z}$ be the composition ${R'}^\bot\hookrightarrow U^*\otimes U^*\overset{r_0^\#}\to\kk$. Since $z+v_1^2$ is a central element, $\widetilde{\theta_z}$ is a Clifford map of $\widetilde{E}$.

The main purpose of this section is to recover Kn\"{o}rrer's Periodicity Theorem in case of quadric hypersurface singularities. Firstly, we recover \cite[Corollary 2.8]{K} and \cite[Theorem 1.6]{CKMW} for quadrics without using matrix factorizations.

\begin{theorem}\label{thmk} Retain the notation as above. Then the algebra $A$ is a noncommutative isolated singularity if and only if so is $A^\#$.
\end{theorem}
\begin{proof} By Theorem \ref{thm2}, $A$ (resp. $A^\#$) is a noncommutative isolated singularity if and only if $E(\theta_z)$ (resp. $\widetilde{E}(\widetilde{\theta_z})$) is a $\mathbb Z_2$-graded semisimple algebra. By Corollary \ref{cort2}, $E(\theta_z)$ is a $\mathbb Z_2$-graded semisimple algebra if and only if so is $\widetilde{E}(\widetilde{\theta_z})$.
\end{proof}

The {\it second double branched cover} of $A$ is defined to be the algebra $$A^{\#\#}:=S[v_1,v_2]/(z+v_1^2+v_2^2).$$
Let $W=V\oplus\kk v_1\oplus\kk v_2$. Then $S[v_1,v_2]=T(W)/(R'')$, where $$R''=R\oplus\{v\otimes v_1-v_1\otimes v|v\in V\}\oplus\{v\otimes v_2-v_2\otimes v|v\in V\}\oplus\mathbb C(v_1\otimes v_2-v_2\otimes v_1).$$

Let $\pi_{S[v_1,v_2]}\colon T(W)\to S[v_1,v_2]$ be the projection map. Let $r_0^{\#\#}=r_0+v_1\otimes v_1+v_2\otimes v_2$. Then $\pi_{S[v_1,v_2]}(r_0^{\#\#})=z+v_1^2+v_2^2$. Let $\widetilde{\widetilde{\theta_z}}$ be the composition ${R''}^\bot\hookrightarrow W^*\otimes W^*\overset{r_0^{\#\#}}\to\kk$. Then $\widetilde{\widetilde{\theta_z}}$ is a Clifford map associated to $\widetilde{\widetilde{E}}$, which is  corresponding to the map obtained in (\ref{eqt4}).

We recover the following Kn\"{o}rrer's Periodicity Theorem (cf. \cite[Theorem 3.1]{K} and \cite[Theorem 1.7]{CKMW}) for quadrics without using matrix factorizations.

\begin{theorem}\label{thmkp} Retain the notation as above. Then there is an equivalence of triangulated categories $\underline{\mcm} A\cong\underline{\mcm} A^{\#\#}$.
\end{theorem}
\begin{proof} By Theorem \ref{thm1}(iii), we have equivalences of triangulated categories $\underline{\mcm} A\cong D^b(\gr_{\mathbb Z_2}E(\theta_z))$ and $\underline{\mcm} A^{\#\#}\cong D^b(\gr_{\mathbb Z_2}\widetilde{\widetilde{E}}(\widetilde{\widetilde{\theta_z}}))$. By Proposition \ref{propt2}, we have an equivalence of abelian categories $\gr_{\mathbb Z_2}E(\theta_z)\cong \gr_{\mathbb Z_2}\widetilde{\widetilde{E}}(\widetilde{\widetilde{\theta_z}})$. Hence the result follows.
\end{proof}

\section{Examples}\label{sec-exam}

In this section, we will list all the possible noncommutative quadric hypersurfaces obtained from an Artin-Schelter regular algebra of type-$S_2$ as listed in \cite[Table 3.11, P.183]{AS}. The key point in the computation is the analysis of the structures of degree zero part of the Clifford deformations.

Let $U$ be a 3-dimensional vector space with a fixed basis $X,Y,Z$, and in this section, we assume
\begin{equation}\label{eq-alg}
  S=\kk \langle X,Y,Z\rangle/(f_1,f_2,f_3),
\end{equation}
where $$f_1=ZX+XZ,f_2=YZ+ZY,f_3=X^2+Y^2.$$
The associated algebra $S$ is indeed the Artin-Schelter regular algebra of type-$S_2$ as listed in \cite[Table 3.11, P.183]{AS}. Then $S$ is a Koszul algebra of global dimension 3, which is also a domain.

Let $V=U^*$ be the dual space. To simplify the notations, we write $x,y,z$ for the basis of $V$ dual to $X,Y,Z$.

The quadratic dual $E(=S^!)$ of $S$ has been computed in Example \ref{ex3}, which is the algebra $E=T(V)/(R)$, where $R$ is the subspace of $V\otimes V$ spanned by $$xz-zx,yz-zy,x^2-y^2,z^2,xy,yx.$$ The basis of $V\otimes R\cap R\otimes V$ has been listed in Example \ref{ex3}.
Now it is easy to check that the only possible nontrivial Clifford maps are defined as following
\begin{equation}\label{eqex1}
  \theta\colon R\to \kk,\ \ \theta(z^2)=\alpha, \theta(xy)=\theta(yx)=\beta,\theta(x^2-y^2)=\lambda\text{ and }\theta(\text{others})=0,
\end{equation}
where $0\neq(\alpha,\beta,\lambda)\in\kk^3$.

Let $C=E(\theta)_0$. Then $C$ is a commutative algebra and has a basis $1,xz,yz,x^2$. Write $a=xz,b=yz,c=x^2$. We have the following table of multiplications of $C$:
$$\begin{tabular}{|c|c|c|c|c|}
  \hline
   & 1 & $a$ & $b$ & $c$ \\
   \hline
  1 & 1 & $a$ & $b$ & $c$ \\
  \hline
  $a$ & $a$ & $\alpha c$ & $\alpha\beta$ & $\lambda a+\beta b$ \\
  \hline
  $b$ & $b$ & $\alpha\beta$ & $\alpha c-\lambda\alpha$ & $\beta a$ \\
  \hline
  $c$ & $c$ & $\lambda a+\beta b$ & $\beta\alpha$ & $\lambda c+\beta^2$ \\
  \hline
  \end{tabular}$$
  \begin{center}  \text{Table 1.}\end{center}

We make a detailed analysis of the structures of $C$ by choosing different scalars $(\alpha,\beta,\lambda)$. Since $C$ is 4-dimensional, $C$ is either semisimple or $C/J$ is isomorphic to a product of $\kk$ where $J$ is the Jacobson radical of $C$.

We may assume $\alpha=1$ and $\beta=1$ (one may replace $z$ by $\sqrt{\alpha}z$, $x$ by $\sqrt{\beta}x$ and $y$ by $\sqrt{\beta}y$, if necessary).

Case (i). $\alpha\neq0$, $\beta\neq0$.

By Table 1, the commutative algebra $C$ has the relations $a^2=c$, $ab=1$, $ac=\lambda a+b$, $b^2=c-\lambda$, $bc=a$, $c^2=\lambda c+1$. One sees that $C$ is generated by $a$. By $c^2=\lambda c+1$, one has $C\cong \kk[a]/(f)$ where $f=a^4-\lambda a^2-1$.

If $\lambda\neq \pm2\sqrt{-1}$, then $f$ does not have multiple roots. Then $C\cong\kk[a]/(f)\cong \kk^4$.

If $\lambda=\pm2\sqrt{-1}$, then $f=(a^2\pm\sqrt{-1})^2$. It follows that $C\cong \kk[u]/(u^2)\times \kk[u]/(u^2)$.

Case (ii). $\alpha=1,\beta=0,\lambda=0$.

The commutative algebra $C$ has relations $a^2=c,b^2=c$ and $ab=ac=bc=c^2=0$. Then $C\cong \kk[u,v]/(uv,u^2-v^2)\cong\kk[u,v]/(u^2,v^2)$.

Case (iii). $\alpha=1,\beta=0,\lambda=1$.

By Table 1, the commutative algebra $C$ has relations $a^2=c$, $ac=a$, $b^2=c-1$, $c^2=c$ $ab=bc=0$. Let $e_1,e_2,e_3,e_4$ be the set of complete primitive idempotents of $\kk^4$. Then $C\cong \kk^4$ by setting $a\mapsto e_1-e_2$, $b\mapsto \sqrt{-1}(e_3+e_4)$, $c\mapsto e_1+e_2$.

Case (iv). $\alpha=0,\beta=1$.

By Table 1, $C$ has the relations $a^2=ab=b^2=0$, $ac=\lambda a+b$, $bc=a$, $c^2=\lambda c+1$.

If $\lambda=\pm2\sqrt{-1}$, then $C\cong\kk[u,v]/(u^2,v^2)$. The isomorphism is defined by $u\mapsto a+\frac{\lambda}{2}b$, $v\mapsto c-\frac{\lambda}{2}$.

Assume $\lambda\neq \pm2\sqrt{-1}$. By the identity $c^2-\lambda c-1=0$, we have $(c-t_1)(c-t_2)=0$ where $t_1=\frac{\lambda+\sqrt{\lambda^2+1}}{2}$ and $t_2=\frac{\lambda-\sqrt{\lambda^2+4}}{2}$. Let $p=\frac{t_1}{t_1t_2-t_1^2}$ and $q=\frac{t_2}{t_1t_2-t_2^2}$. Then $e_1=p(c-t_1)$ and $e_2=q(c-t_2)$ are idempotents such that $e_1+e_2=1$ and $e_1e_2=0$. Then $C\cong \kk[u]/(u^2)\times \kk[u]/(u^2)$. The isomorphism is defined by $(u,0)\mapsto a+\frac{\lambda-\sqrt{\lambda^2+4}}{2}b$, $(0,u)\mapsto a+\frac{\lambda+\sqrt{\lambda^2+4}}{2}b$.

Case (v). $\alpha=\beta=0,\lambda=1$.

By Table 1, one has $a^2=ab=bc=b^2=0$, $ac=a$ and $c^2=c$. Then $c$ and $1-c$ are primitive idempotents. It follows that $C\cong \kk[u]/(u^2)\times \kk[u]/(u^2)$.

Summarizing, we have the following table of all the possible noncommutative quadric hypersurfaces defined by a central element $w\in S_2$ of $S$ (recall that a Clifford map $\theta$ corresponds to a central element in $S_2$, see Lemma \ref{lem6}), in which $(\alpha,\beta,\lambda)$ is the scalar of the Clifford map as defined in (\ref{eqex1}).

{\small
\begin{tabular}{|c|c|c|c|c|}
\hline
No. &$(\alpha,\beta,\lambda)$ & $w$ &$C=E(\theta)_0$ & singularities of $S/wS$ \\
  \hline
1&$(1,1,\lambda\neq \pm2\sqrt{-1})$ & $Z^2+XY+YX+\lambda X^2$ & $\kk^4$ & isolated \\
  \hline
2&$(1,1,\pm2\sqrt{-1})$ & $Z^2+XY+YX\pm2\sqrt{-1} X^2$ & $\kk[u]/(u^2)\times\kk[u]/(u^2)$ & nonisolated \\
  \hline
3&$(1,0,0)$ & $Z^2$ & $\kk[u,v]/(u^2,v^2)$& nonisolated \\
  \hline
4&$(1,0,1)$ & $Z^2+X^2$ & $\kk^4$ &isolated\\
    \hline
5&$(0,1,\lambda\neq \pm2\sqrt{-1})$&$XY+YX+\lambda X^2$&$\kk[u]/(u^2)\times\kk[u]/(u^2)$& nonisolated\\
  \hline
6&$(0,1,\pm2\sqrt{-1})$&$XY+YX\pm2\sqrt{-1}X^2$&$\kk[u,v]/(u^2,v^2)$&nonisolated\\
  \hline
7&$(0,0,1)$&$X^2$&$\kk[u]/(u^2)\times\kk[u]/(u^2)$&nonisolated\\
  \hline
\end{tabular}
  \begin{center}Table 2.\end{center}}

Note that the noncommutative quadrics of Nos. 1 and 4 in the above table are isolated singularities. We remark that the associated Clifford deformations of these two cases are both isomorphic to $\mathbb {C}\mathbb{G}^{\times 4}$. This is because the associated Clifford deformations are always strongly $\mathbb Z_2$-graded and commutative.

\vspace{5mm}

\subsection*{Acknowledgments}
The authors thank James J. Zhang and Guisong Zhou for many useful conversations. J.-W. He was supported by NSFC (No. 11971141). Y. Ye was supported by NSFC (No. 11571329, 11971449).

\vspace{5mm}


\end{document}